\documentclass[leqno,12pt]{amsart} 
\usepackage{amssymb} 
 
\theoremstyle{plain} 
\newtheorem{thm}{\indent\sc Theorem}[section] 
\newtheorem{lemma}[thm]{\indent\sc Lemma}

\newtheorem{proposition}[thm]{\indent\sc Proposition}

\theoremstyle{definition}

\theoremstyle{remark}
\newtheorem{rem}{\indent\sc Remark}[thm]
\newcommand{\bC}{{\mathbb C}}

\newcommand{\bR}{{\mathbb R}}

\newcommand{\rRe}{{\mathrm {Re}}}

\usepackage{graphicx}

\begin{document}

\title[mean curvature flow in submanifolds]{mean curvature flow in submanifolds} 

\author[H. Nakahara]{Hiroshi Nakahara} 
\subjclass[2010]{53A07 Primary, 53A10, 53C42(secondary).
}

\address{
Department of Mathematics \endgraf 
Tokyo Institute of Technology \endgraf
 2-21-1  O-okayama, Meguro, Tokyo\endgraf
Japan
}
\email{12d00031@math.titech.ac.jp}

\maketitle

\begin{abstract}
We obtain explicit solutions of the mean curvature flow in some submanifolds of the Euclidean space.
We give particularly an explicit solution of the flow of a hypersurface in the Lagrangian self-expander $L$ which is constructed in the article of Joyce, Lee and Tsui and show that it converge to a minimal one.
\end{abstract}
\section{\rm{Introduction}}
\quad Mean curvature flow evolves the submanifolds of the riemannian manifolds in the direction of their mean curvature vectors. The short-time existence and uniqueness of the solution of the mean curvature flow equation was proved. The mean curvature vector of a submanifold in a riemannian manifold is $\sum_j (\nabla_{e_j}e_j)^\perp$ where $\nabla$ is the Levi-Civita connection, $\{e_j \}_j $ is a orthonormal frame of the submanifold's tangent bundle and $\perp$ is the orthogonal projection to the normal bundle of it. 
The mean curvature flow is an important example of the changing. Sometimes the flow stops because of some singular points. Recently the finite time singularity is focused and mean curvature flow has been investigated since it appeared from the study of annealing metals in physics in 1956. Many researchers study it and have found a lot of results. 
The mean curvature flow is the steepest descent flow for the area functional and is described by a parabolic system of partial differential equations. 
The problem area of it is geometry but it implies questions of partial differential equations. 
When $M_0$ is a hypersurface in $\bR^{n+1}$ and  $\{M_t \}_{t\in [0,\epsilon)}$ is the solution of mean curvature flow, then, by the weak maximum principle of it (See also \cite{Ecker}), we can see that if the initial manifold $M_0$ is 
in an open ball $B(0,r),$ where $r>0,$ then $M_t \subset B(0, \sqrt{r^2 -2nt}),$ for any $t\in [0,\epsilon).$ 
Furthermore, other properties of the mean curvature flow in $\bR^N$ have been extensively studied.
For example, Wang investigates the mean curvature flow of graphs in \cite{Wang} and the author constructs explicit self-similar solutions and translating solitons for the mean curvature flow in $\bC ^n (=\bR^{2n})$ in his last article \cite{N}. 
Since the definition of the mean curvature vector is intricate as well as abstract, it is difficult or impossible to find the non-trivial and explicit solution of a given initial submanifold. 
However, we think of how explicit submanifolds of some concrete manifolds move by the mean curvature flow.   
The author noticed that if we consider the mean curvature flow of the hypersurfaces $\{s= constant\}$ in the Lagrangian submanifold of the form of \cite[Ansatz 3.1]{JLT}, they may change inside themselves. The prediction is correct with some restriction, to be showed in this paper. In this short article we deal with the mean curvature flow in the Lagrangian submanifold and give the explicit solutions. 
By the next section of this paper we get the following Theorem \ref{m} and we treat a self-similar solution for the mean curvature flow in there. 
Note that we can learn self-similar solutions for the mean curvature flow elementarily from \cite{JLT} and \cite{N}, and 
the self-expander, which means a self-similar solution of expanding, in the articles is an explicit product manifold $\mathcal {S}^{n-1} \times \bR$ in $\bC^n .$
Thus in the next steps, it is natural to study the mean curvature flow of the sphere $\mathcal{S}^{n-1}$ inside it and Theorem \ref{m}. 
\begin{thm}\label{m}
Let $a>0,$ $E\geq1$ and $\alpha\geq 0$ be constants.
Define $r:[0,\infty) \to \bR$ by 
$r(s)=\sqrt{1/a+s^2}$ and $\phi_E:[0,\infty) \to \bR$ by 
$$\phi _E(s)=\int_0 ^s \frac{t\,dt}{(1/a+t^2)\sqrt{E(1+at^2)^n e^{\alpha t^2}-1}}.$$ 
Set $$l_s =\{(x_1 r(s)e^{i\,\phi_E (s)},\ldots ,x_n r(s)e^{i\,\phi_E (s)});\sum_{j=1}^n x_j ^2 =1,\,x_1,\ldots ,x_n \in \bR\},$$ for $s\in [0,\infty),$ and 
\begin{equation}\label{s}
L=\bigcup _{s\in [0,\infty)}l_s.
\end{equation}
$($Then, clearly, $l_s \subset L\subset \bC^n .)$
Fix $s_0 \in (0,\infty).$ Suppose that $f$ is the solution of the following initial-value problem
\begin{equation}
\left\{
\begin{aligned}
\frac{df}{dt}&=-(n-1)\cdot \frac{E(1+af^2)^n -e^{-\alpha f^2}}{Es(1+af^2)^n}\\
 f(0)&=s_0
\end{aligned}
\right.
\end{equation}
Then $\{\l_{f(t)}\}_t $ is a mean curvature flow in $L.$ In addition, when $E=1,$ then $L$ is the Lagrangian self-expander 
in \cite[Theorem C]{JLT} and the domain of definition of $f$ can be extended to $[0,\infty ),$ $\lim _{t\to \infty}f(t) =0$ and $l_0 $ is a minimal hypersurface in $L.$
\end{thm}
\section{\rm{Results and Proofs}}
In order to discuss the mean curvature flow in submanifolds, firstly, we consider the following well known Proposition. 
Note that, in this article, when a manifold $M$ is a submanifold in a riemannian manifold $N,$ then we denote $A_{M,N}$ the second fundamental form of $M$ in $N$ and $\nabla^{N}, \nabla^{M}$ the Levi-Civita connections on $N$ and $M$ respectively. Hence $A_{M,N}\in C^{\infty}(M,(TN/TM)\otimes T^* M \otimes T^* M ).$
\begin{proposition}\label{prop}
Let $l,$ $L$ be submanifolds in $\bC^n .$ Suppose that $l$ is a submanifold in $L.$
Put $H$ to be the mean curvature vector of $l$ in $L,$ and $\bar{H}$ to be the mean curvature vector of $l$ in $\bC^n.$ Fix $p\in l .$ Then $$H(p)=\bar{H}(p)-\sum_{j} A_{L,\bC^n} (e_j ,e_j)$$where $\{e_j\}_j$ is an orthonormal basis of $T_p l.$ Hence we can see that $$H(p)=\pi _{T_p L}(\bar{H}(p)),$$ where $\pi_{T_p L}(\bar{H}(p))$ is the orthogonal projection of $\bar{H}(p)$ to $T_p L.$
\end{proposition}
The reader can try to prove Proposition \ref{prop} or skip the proof below if it is already done. 
\begin{proof}
From the definitions of the mean curvature vector and the second fundamental form we have 
\begin{equation*}
\begin{split}
H(p)=&\sum_j A_{l,L}(e_j ,e_j)
  =\sum_j(\nabla_{e_j}^L e_j -\nabla_{e_j}^l e_j)
  =\sum_j(\nabla_{e_j}^{\bC^n}e_j -A_{L,\bC^n}(e_j ,e_j)-\nabla_{e_j}^l e_j)\\
  =&\sum_j(A_{l,\bC^n}(e_j ,e_j)-A_{L,\bC^n}(e_j ,e_j))
  =\bar{H}(p)-\sum_j A_{L,\bC^n} (e_j ,e_j). 
\end{split}
\end{equation*}
This finishes the proof.
\end{proof}
In the following Theorem \ref{k}, from a direct calculation, the submanifolds $L$ are Lagrangian submanifold.
\begin{thm}\label{k}
Let $I$ be an interval of $\bR $ and $w :I\to \bC \setminus \{0\}$ be a smooth function. Suppose that $\dot{w}(s)\neq 0,$ for any $s\in I.$
Define submanifolds $l_s ,$ for $s\in I,$ in $\bC^n$ by $$l_s =\{(x_1 w(s),\ldots ,x_n w(s));\sum_{j=1}^n x_j ^2 =1,\,x_1,\ldots ,x_n \in \bR\},$$ and submanifold L in $\bC^n$
by $$L=\bigcup _{s\in I}l_s.$$
$($Clearly, $l_s \subset L\subset \bC^n .)$
Let $H_{s}$ be the mean curvature vector of $l_{s}$ in $L.$
Then 
\begin{equation}\label{h}
H_{s}(x_1 w(s),\ldots ,x_n w(s))=-\frac{(n-1)\rRe \,(\bar{w}(s)\dot{w}(s))}{|w(s)|^2|\dot{w}(s)|^2}\cdot\frac{\partial }{\partial s}
\end{equation}
holds, where $\partial /\partial s =(x_1 \dot{w}(s),\ldots ,x_n \dot{w}(s))\in T_{(x_1 w(s),\ldots ,x_n w(s))}L.$
Thus, by the definition of the mean curvature flow, if we suppose that $f$ is a solution of the following ordinal differential equation $$\frac{df (t)}{dt}= -\frac{(n-1)\rRe \,(\bar{w}(f(t))\dot{w}(f(t)))}{|w(f(t))|^2|\dot{w}(f(t))|^2},$$ then 
$\{l_{f(t)}\} _t$ is a mean curvature flow in $L.$
\end{thm}
We notice that the following Lemma \ref{S} holds and it is a lemma of Theorem \ref{k}.  
\begin{lemma}\label{S}
Let $\alpha \in \bC \setminus \{0\}$ be a constant. Define a submanifold $S$ in $\bC ^n$ by 
$$S=\{\alpha(x_1 ,\ldots ,x_n)\in \bC^n;\sum_{j=1}^n x_j ^2=1,\,x_1,\ldots ,x_n \in \bR \}.$$ Fix $p\in S.$
Then $$H(p)=-\frac{n-1}{|\alpha|^2}p,$$
where $H(p)$ is the mean curvature vector of $S$ at $p.$ 
\end{lemma}
\begin{proof}
Let $\{e_1 ,\ldots ,e_{n-1}\}$ be an orthonormal basis of $T_p S.$ 
Let $V_j$ be the plane which is generated by $e_j$ and $\overrightarrow{Op},$ where $O=(0,\ldots ,0)\in\bC^n.$ Since the intersection of $S$ and $V_j$ is a circle of radius $|\alpha|$ with center $O,$
we can get curves $c_1 ,\ldots ,c_{n-1}:\bR\to S$ such that 
$$c_j (0)=p,\quad \dot{c}_j (0)=e_j,\quad \ddot{c}_j (0)=-\frac{1}{|\alpha|^2}p,$$ for any $j.$ 
We compute 
\begin{equation*}
H(p)=\sum_{j=1}^{n-1}A_{S,\bC^n}(e_j ,e_j)
  =\sum_{j=1}^{n-1}\left(\nabla _{e_j}^{\bC^n}e_j \right)^{\perp}
  =\sum_{j=1}^{n-1}(\ddot{c}_j (0))^{\perp}
  =\sum_{j=1}^{n-1}\left(-\frac{1}{|\alpha|^2}p\right)^{\perp}
  =-\frac{n-1}{|\alpha|^2}p,
\end{equation*}
where $\perp$ is the orthogonal projection to $T_p ^{\perp} S.$ This completes the proof.
\end{proof}
Now we prove Theorem \ref{k}.

{\it proof of Theorem \ref{k}.}
We denote by $\bar{H}_s$ the mean curvature vector of $l_s$ in $\bC^2.$ Fix $p=(x_1 w(s),\ldots ,x_n w(s))\in l_s.$ By Lemma \ref{S}, 
$$\bar{H}_s (p)=-\frac{n-1}{|w (s)|^2}(p).$$
By Proposition \ref{prop}, we have 
\begin{equation*}
H(p)= \pi _{T_p L}(\bar{H}(p))
=-\frac{n-1}{|w (s)|^2}\cdot \pi _{T_p L}(p)
\end{equation*}
From a direct calculation, we can see $\partial /\partial s \perp T_p l_s.$ Hence we obtain
\begin{equation*}
H(p)=-\frac{n-1}{|w (s)|^2}\cdot \frac{p \cdot \partial /\partial s}{\partial /\partial s \cdot \partial /\partial s}\cdot \frac{\partial}{\partial s}
=-\frac{(n-1)\rRe \,(\bar{w}(s)\dot{w}(s))}{|w(s)|^2|\dot{w}(s)|^2}\cdot\frac{\partial }{\partial s}  .
\end{equation*}
This finishes the proof.
\qed

Next we consider the following Remark \ref{r} and Figure 1. If we put $w_1 = \cdots =w_1$ in the construction of the Lagrangian self-expander given by Joyce, Lee and Tsui \cite[Thorem C]{JLT}, then we can find a minimal hypersurface in the self-expander. 
\begin{figure}[!ht]
\begin{center}
\includegraphics[keepaspectratio, scale=0.4]
{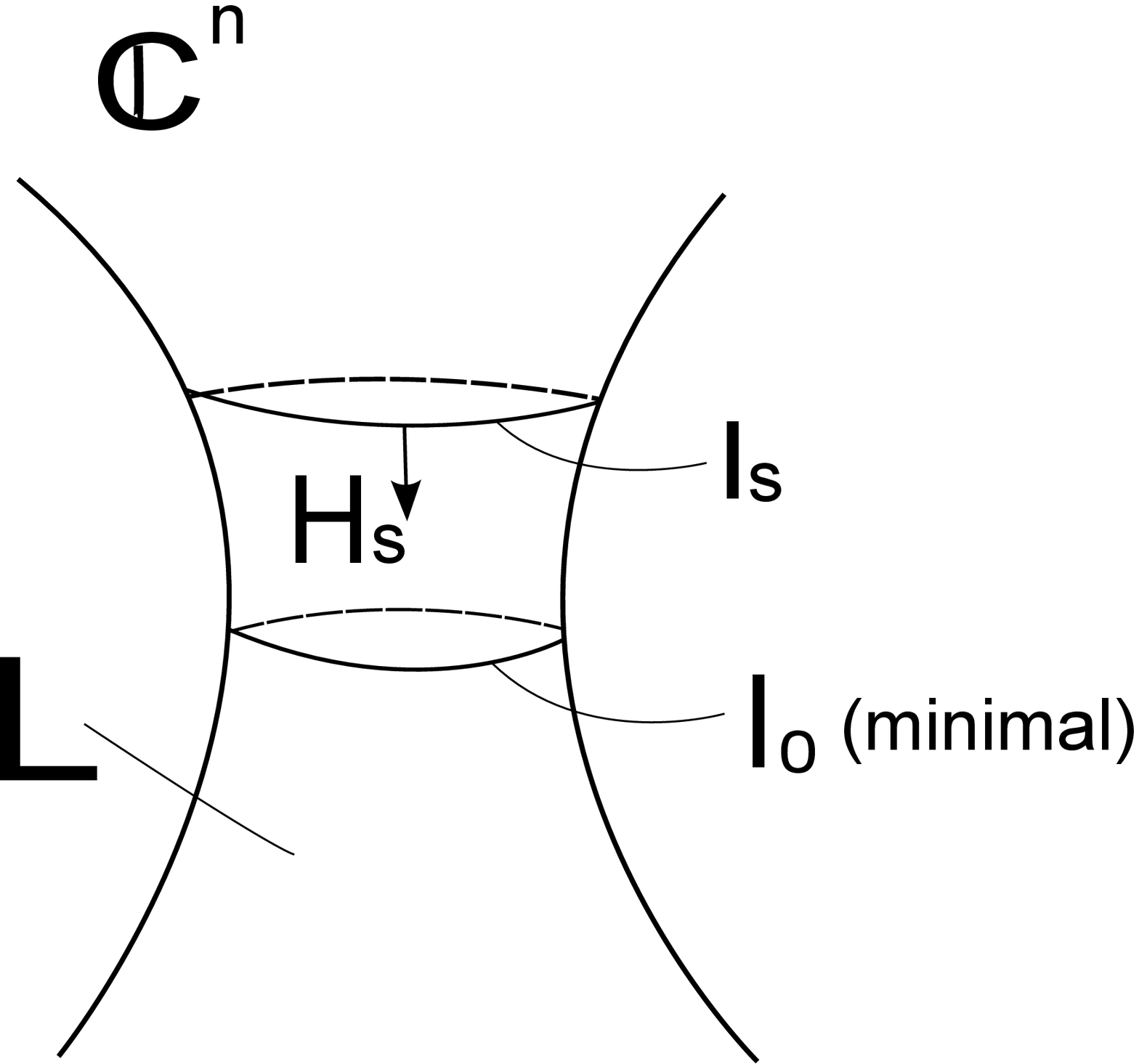}
\caption{Remark \ref{r}}
\label{ddd}
\end{center}
\end{figure}
\begin{rem}\label{r}
Let $a>0$ and $\alpha\geq 0$ be constants.
Define $r:\bR \to \bR$ by 
$r(s)=\sqrt{1/a+s^2}$ and $\phi:\bR \to \bR$ by 
$$\phi(s)=\int_0 ^s \frac{|t|dt}{(1/a+t^2)\sqrt{(1+at^2)^n e^{\alpha t^2}-1}}.$$
In the situation of Theorem \ref{k}, if we put $I=\bR$ and $w(s)=r(s)e^{i\phi(s)},$ then $L$ is the Lagrangian self-expander constructed in \cite[Theorem C]{JLT}. Then we compute 
\begin{equation}\label{r}
\begin{split}
\frac{\rRe \,(\bar{w}(s)\dot{w}(s))}{|w(s)|^2|\dot{w}(s)|^2}=&\frac{\rRe \,\left(r(s)e^{-i\phi (s)}(\dot{r}(s)e^{i\phi(s)}+ir\dot{\phi}(s)e^{i\phi(s)})\right)}{r(s)^2\cdot |\dot{r}(s)e^{i\phi (s)}+ir(s)\dot{\phi}(s)e^{i\phi (s)}|^2}\\
=&\frac{r(s) \dot{r}(s)}{r(s)^2\cdot |\dot{r}(s)+ir(s)\dot{\phi}(s)|^2}\\
=&\frac{r(s) \dot{r}(s)}{r(s)^2 \dot{r}(s)^2+r(s)^4 \dot{\phi}(s)^2}\\
=&\frac{s}{s^2+s^2 /((1+at^2)^n e^{\alpha s^2}-1)}\\
=&\frac{1}{s+s /((1+at^2)^n e^{\alpha s^2}-1)}\\
=&\frac{1}{s(1+as^2)^n e^{\alpha s^2}/((1+as^2)^n e^{\alpha s^2}-1)}\\
=&\frac{(1+as^2)^n e^{\alpha s^2}-1}{s(1+as^2)^n e^{\alpha s^2}}\\
=&\frac{(1+as^2)^n -e^{-\alpha s^2}}{s(1+as^2)^n}.\\
\end{split}
\end{equation}
By the equation \eqref{h}, the last computation and L'H$\hat{\rm{o}}$pital's rule, we obtain
\begin{equation}\label{ee}
\begin{split}
H_0 (x_1 w(0),\ldots ,x_n w(0))=&-(n-1)\cdot \frac{\rRe \,(\bar{w}(0)\dot{w}(0))}{|w(0)|^2|\dot{w}(0)|^2}\cdot \frac{\partial }{\partial s}\\
=&-(n-1)\cdot \lim_{s\to 0}\frac{(1+as^2)^n -e^{-\alpha s^2}}{s(1+as^2)^n}\cdot \frac{\partial }{\partial s}\\
=&-(n-1)\cdot \lim_{s\to 0}\frac{n(1+as^2)^{n-1}\cdot 2as+2\alpha s e^{-\alpha s^2}}{(1+as^2)^n +s\cdot n(1+as^2)^{n-1}\cdot 2as}\cdot \frac{\partial}{\partial s}\\
=&\vec{0}.
\end{split}
\end{equation}
Therefore, in this case, $l_0$ is minimal in $L.$ Secondly, the author is going to prove a general version of the fact $l_0$ is volume-minimizing as well as minimal in his next paper. See also it. A bit of information is below. 
We can see $\pi ^*({\rm vol} _{l_0})$ is a calibration of $L$ that is described in the section 4 of \cite{J} and that is a little difficult to prove, and $l_0$ is the calibrated submanifold, where $\pi $ is a projection from $L$ to $l_0$ defined by 
$$\pi (x_1 w_1(s),\ldots ,x_n w_n (s))=(x_1 w_1(0),\ldots ,x_n w_n (0))$$ 
and 
${\rm vol} _{l_0}$ is the volume form of $l_0$ with respect to the induced metric of the Euclidean metric in $\bC^n .$ Note that $\pi$ is well-defined. 
This implies that $l_0$ is volume-minimizing as well as minimal. 
In addition, without the restriction $w_1 =\cdots = w_n ,$the submanifold $\{y=0\}$ is a calibrated manifold in the self-expander $L$ \cite[Thorem C]{JLT}.
\end{rem}
Further we have to research the following situation. 
\begin{rem}\label{r2}
Let $a>0,$ $E>1$ and $\alpha\geq 0$ be constants.
Define $r:(0,\infty) \to \bR$ by 
$r(s)=\sqrt{1/a+s^2}$ and $\phi_E:(0,\infty) \to \bR$ by 
$$\phi _E(s)=\int_0 ^s \frac{t\,dt}{(1/a+t^2)\sqrt{E(1+at^2)^n e^{\alpha t^2}-1}}.$$
In the situation of Theorem \ref{k}, if we put $I=(0,\infty)$ and $w(s)=r(s)e^{i\phi_E (s)},$ then $L$ is the Lagrangian self-similar solution constructed in \cite[Theorem 1.3]{N}. Now we write 
\begin{equation}\label{t}
\begin{split}
\frac{\partial}{\partial s}=&(x_1\dot{w} \ldots ,x_n \dot{w})\\
=&\dot{w}\cdot (x_1, \ldots ,x_n)\\
=&(\dot{r}e^{i\phi_E} +r\,i\dot{\phi}_E e^{i\phi _E})\cdot (x_1, \ldots ,x_n)\\
=&(\dot{r} +r\,i\dot{\phi}_E )\cdot e^{i\phi _E} \cdot (x_1, \ldots ,x_n).\\
\end{split}
\end{equation}
From \eqref{r} and \eqref{t}, we obtain 

\begin{equation*}
\begin{split}
&H_s (x_1 w(s),\ldots ,x_n w(s))
=-(n-1)\frac{\rRe \,(\bar{w}(s)\dot{w}(s))}{|w(s)|^2|\dot{w}(s)|^2}\cdot \frac{\partial}{\partial s} \\
=&\frac{-(n-1)r\dot{r}}{r^2\cdot |\dot{r}+ir\dot{\phi_E}|^2}\cdot (\dot{r} +r\,i\dot{\phi}_E )\cdot e^{i\phi _E} \cdot (x_1, \ldots ,x_n)\\
=&\frac{-(n-1)r\dot{r}}{r^2\cdot( \dot{r}-ir\dot{\phi_E})}\cdot e^{i\phi _E} \cdot (x_1, \ldots ,x_n)\\
=&\frac{-(n-1)s}{1/a +s^2}\cdot \left(\frac{s}{\sqrt{1/a +s^2}}-i \frac{s}{\sqrt{(1/a +s^2)\{E(1+a  s^2)^n  e^{\alpha s^2}-1\}}} \right)^{-1}\cdot e^{i\phi _E} \cdot (x_1, \ldots ,x_n)\\
=&\frac{-(n-1)}{1/a +s^2}\cdot \left(\frac{1}{\sqrt{1/a +s^2}}-i \frac{1}{\sqrt{(1/a +s^2)\{E(1+a  s^2)^n  e^{\alpha s^2}-1\}}} \right)^{-1}\cdot e^{i\phi _E} \cdot (x_1, \ldots ,x_n).\\
\end{split}
\end{equation*}
Therefore 
\begin{equation}\label{y}
\begin{split}
\lim_{s\to +0} H_{s}(x_1 w(s),\ldots ,x_n w(s))=& -(n-1)a\cdot \left(\sqrt{a}-i \sqrt{\frac{a}{E  -1}} \right)^{-1}\cdot (x_1, \ldots ,x_n)\\
=&-(n-1)a\cdot \frac{\sqrt{a}+i \sqrt{a/(E  -1)} }{a+a/(E-1)}\cdot (x_1, \ldots ,x_n)\\
=&-(n-1)\cdot \frac{\sqrt{a}+i \sqrt{a/(E  -1)} }{1+1/(E-1)}\cdot (x_1, \ldots ,x_n)\\
=&-(n-1) \sqrt{a}\cdot \,\,\frac{1+i \sqrt{1/(E  -1)} }{1+1/(E-1)}\cdot (x_1, \ldots ,x_n)\\
=&-(n-1) \sqrt{a}\cdot \,\, \frac{E-1+i \sqrt{E  -1} }{E}\cdot (x_1, \ldots ,x_n).\\
\end{split}
\end{equation}
By \eqref{y}, we can check 
$$\lim_{E\to 1+0}\left(\lim_{s\to +0} H_{s}(x_1 w(s),\ldots ,x_n w(s))\right)=\vec{0}$$
which matches \eqref{ee}.
\end{rem}
\section{\rm{Discussion}}
We can also consider the mean curvature flow in product manifolds, for example, cone manifolds, the paraboloid of revolution without the bottom point and so on, similarly to this paper and can obtain the solutions.   

\end{document}